\documentclass[12pt]{article}
\usepackage[T1]{fontenc}
\usepackage[left=25mm,right=25mm,top=25mm, bottom=25mm]{geometry}
\usepackage{amssymb,amsmath,amscd, amsthm,stmaryrd,verbatim,
	mathrsfs,tikz,color , amsfonts}
\usepackage{cite}
\usepackage{setspace} 
\allowdisplaybreaks
\usepackage[vcentermath,enableskew,noautoscale]{youngtab}
\usepackage{young}
\usepackage{hyperref}

\newtheorem{lem}{Lemma}[section]
\newtheorem{them}[lem]{Theorem}
\newtheorem{Def}[lem]{Definition}
\newtheorem{prop}[lem]{Proposition}
\newtheorem{cor}[lem]{Corollary}

\title{\textbf{The Stability of Pointwise Hyperbolic Systems}}
\author{ Haiye Guo, Yunhua Zhou 
	\\
\small\textit{College of Mathematics and Statistics, Chongqing University, Chongqing 401331, China}}
\date{}
\begin{document}
	\bibliographystyle{plain}
	\maketitle
	\renewcommand{\abstractname}{}
	\begin{abstract}
		\noindent\textbf{Abstract:}	The stability of the system is an important part of the research on differential dynamical systems. This paper considers a pointwise hyperbolic system defined on a connected open subset N of a compact smooth Riemannian manifold M. The hyperbolicity may weaken when approaching the boundary of the open set. By analogy with the stability of hyperbolic systems, this paper constructs the expansive property and the shadowing lemma on the pointwise pseudo orbits and thus obtains the stability of pointwise hyperbolic systems. 

	\end{abstract}
	
\section{Introduction}
Hyperbolic systems are an important part of the study of differential dynamical systems, and so far hyperbolic systems have developed rich theoretical results. For the introduction to hyperbolic systems, see \cite{Sun,Wun,hyperbolic1,hyperbolic2,hyperbolic3} for details. The pointwise hyperbolic system, as a new type of system, is also an important object in the study of differential dynamical systems.
 Based on the uniform compression and expansion, the hyperbolic system is also called uniform hyperbolic system. For pointwise hyperbolic systems, there are compression and expansion rates at each point that depend on the point.
An early pointwise hyperbolic system is the Katok map \cite{Katok}, defined by Katok in 1979. There is also a class of dissipative pointwise hyperbolic systems “Almost Anosov differential homeomorphisms” proposed by Hu and Young in 1995 in \cite{Almost}.
In 2020, Zhou, Chen and Hu studied the pointwise hyperbolic system \cite{SRB} defined on an invariant open set $N$ in a compact smooth Riemannian manifold. Since $N$ is an open set, the hyperbolicity usually becomes weaker near the boundary of $N$, in which case the hyperbolicity is not uniform. \cite{SRB} constructs locally stable and unstable manifolds on the orbits of the points by using the graph transform under certain assumptions and the size of the unstable manifolds depends on the distances of the points to the boundary.
In 2024, in literature \cite{Wu}, Zhou and Wu  continued to follow the similar assumptions, constructed locally stable and unstable manifolds on recurrent-pointwise-pseudo-orbits, and constructed Markov splitting and symbolic models by using the corresponding shadowing lemma. Some of the conclusions and lemmas in the paper are derived from \cite{Wu} or are generalizations of the corresponding ones in \cite{Wu}. Based on these studies of pointwise hyperbolic systems, we hope to generalize some important results in hyperbolic systems to pointwise hyperbolic systems, such as the stability of hyperbolic systems.

The stability theory is an important part of the research on dynamical systems. In 1967, Anosov proved that the Anosov differential homomorphisms on compact Riemannian manifolds with full space as hyperbolic set are $C^1$ structurally stable \cite{Anosov}.
In 1969, Walters proved that the Anosov differential homomorphisms are topologically stable \cite{Peter2}. In 1978, Walters further proved in the literature \cite{Peter} that the expandability of the homomorphisms under the compact metric space with the pseudo orbit shadowing property ensures the topological stability, and in particular, if the perturbation g is also expansive, then the conjugate map can be a homomorphism.
Many scholars have also done a great deal of important work on the study of the stability of dynamical systems. For related details, see \cite{stability1,stability2,stability3,stability4,stability5,stability6,stability7,stability8}.
It is natural for us to consider the stability of pointwise hyperbolic systems, which is of great importance in enriching the study of differential dynamical systems. 

By analogy with the stability of hyperbolic systems, this paper generalizes the expansive property and the shadowing property in hyperbolic systems to pointwise hyperbolic systems. For the general content of pseudo orbits and the shadowing lemma, see \cite{gz1, gz2}.The pseudo orbit shadowing property of a system is helpful for understanding the behavior of the system when it is subjected to small perturbations. Therefore, the definition of pseudo orbits in pointwise hyperbolic systems is a key point in proving the stability. In fact, the literature \cite{Wu} constructs the directed graph by using a countable subset $V$ of the open set $O$ as the vertices, and already defines recurrent-pointwise-pseudo-orbits with the corresponding shadowing lemma. For the need of subsequent proof of stability, we have made some modifications to the definition of recurrent-pointwise-pseudo-orbits, and redefined the pointwise pseudo orbits. 
We call $\{x_n\}_{n\in \mathbb{Z}}\subset N$ a $\varepsilon$-pointwise pseudo orbit of $f$ if the $\{x_n\}_{n\in \mathbb{Z}}\subset N$ satisfies:

(1)\quad $d(f(x_n),x_{n+1})<\delta^u(x_n)Q(x_n)^2Q(f(x_n))$ 

(2)\quad$d(f^{-1}(x_{n+1}),x_n)<\delta^s(x_{n+1})Q(x_{n+1})^2Q(f^{-1}(x_{n+1}))$

(3)\quad$\prod_{i=0}^\infty\|Df|_{E^s(x_{-i})}\|=0,
\prod_{i=0}^\infty\|Df^{-1}|_{E^u(x_i)}\|=0.$\\
At the same time, we also made some modifications to the initial Assumption US in \cite{Wu}, so that Proposition 5.2 in \cite{Wu} still holds for the re-defined pointwise pseudo orbits, and  the corresponding shadowing lemma exists. Then, we defined an expansive property similar to that in hyperbolic systems and demonstrated that under certain assumptions the pointwise hyperbolic system is "pointwise expansive".
On this basis, following the idea of Walters in \cite{Peter}, we obtain the main conclusion of this paper, that is, the stability of the pointwise hyperbolic system $f:M\to M$ satisfying Assumptions U, S, and R under sufficiently small perturbations:

For any sufficiently small $\varepsilon>0$, if the diffeomorphism $g:M\to M$ satisfies\\
(1) $f|_{M\setminus N}=g|_{M\setminus N}$;\\
(2) $\prod_{i = 0}^{\infty}m(Df|_{E^u(g^ix)})=\infty$, $\prod_{i = 0}^{\infty}m(Df^{-1}|_{E^s(g^{-i}x)})=\infty$ for all $x\in N$;\\
(3) $\max\{d(f(x),g(x)),\|D_xf - D_xg\|\}<\xi(x)\delta^u(x)Q(x)^2Q(f(x))$ for all $x\in N$.\\
Then there exists a continuous surjective map $h:M\to M$ such that $h\circ g = f\circ h$. In particular, $h|_{M\setminus N}=id$ and for all $x\in N$, $d(h(x),x)\leq Q(x)$. In this case, $f$ is called "semi - pointwise quasi - stable".
In fact, based on the preservation of pointwise hyperbolicity under perturbations under Assumption $K$, it can be further shown that when $f$ satisfies Assumption $K$, if there exists a sequence of functions $\{r_g^n(x)\}_{n\in\mathbb{Z}}$ such that "for all $n$, when $d(g^nx,g^ny)\leq r_g^n(x)$, then $x = y$", then $h$ is a homeomorphism. In this case, $f$ is called "pointwise quasi - stable". 

 \section{Basic definitions and assumptions}
Let  $ M $  be a  $ C^\infty $  compact and connected Riemannian manifold, and  $ N $  be a connected open subset of  $ M $ . Without loss of generality, assume that the diameter of  $ M $  is less than 1. 
Let  $ f: M \to M $  be a  $ C^{1+\alpha} $  diffeomorphism such that  $ f(N) = N $. Denote by $ \|D_x f \| $  and  $ m(D_x f) $ the norm and the minimum norm of  $ D_x f $, respectively.
\begin{Def}
	$ f \in \text{Diff}^r (M) $  ( $ r \geq 1 $ ) is  called pointwise hyperbolic on an invariant set  $ N \subset M $  if there exist functions  $ \tau^s, \tau^u: N \to \mathbb{R}_{+} $  and an invariant splitting  $ T_{N}M = E^{s} \oplus E^{u} $ , such that for any  $ x \in N $, we have  $ \tau^s(x) < 1 $,  $\tau^u(x) > 1$, and
	$$
	\|D_{x}f|_{E^{s}(x)}\| \leq \tau^{s}(x) < 1,
	$$ 
	$$1 < \tau^{u}(x) \leq m(D_{x}f|_{E^{u}(x)}).
	$$ 
	where $ E^{s}(x) $  and  $ E^{u}(x) $  are the stable and unstable subspaces, respectively.
	
\end{Def}

For pointwise hyperbolic systems, hyperbolicity may weaken near the boundary of  the open set  $N$ , that is, as  $x \to \partial N$, it is possible that  $\|D_x f|_{E^s(x)}\| \to 1$ or $m(D_x f|_{E^u(x)}) \to 1$ . Since $N$ is not compact, in this case, the pointwise hyperbolic system no longer has a uniform contraction and expansion rate on  $N$ .
\\
\textbf{Assumption U:}
(i) For any $x\in N$, it holds that $$  \prod_{i=0}^\infty m(Df|_{E^u(f^ix)})=\infty , \quad \prod_{i=0}^\infty m(Df|_{E^u(f^{-i}x)})=\infty.$$

(ii) There exist $r_0^u,\beta^u>0, \gamma^u>\max\{1,\frac{\beta^u}{\alpha}\}$ and $C^u>1 $ such that for any $ x\in N$ satisfying $d(x,\partial N)\leq r_0^u$, we have:
$$m(Df|_{E^u(x)})-1\geq C^u\max\bigg\{d(x,\partial N)^{\beta^u},\bigg(\frac{d(f(x),\partial N)}{d(x,\partial N)}\bigg)^{\gamma^u}-1\bigg\}$$\\
\textbf{Assumption S:}
(i)For any $x\in N$, it holds that $$  \prod_{i=0}^\infty m(Df^{-1}|_{E^s(f^ix)})=\infty ,  \quad\prod_{i=0}^\infty m(Df^{-1}|_{E^s(f^{-i}x)})=\infty.$$

(ii) There exist $r_0^s,\beta^s>0, \gamma^s>\max\{1,\frac{\beta^s}{\alpha}\}$ and $C^s>1 $ such that  for all $ x\in N$ satisfying $d(x,\partial N)\leq r_0^s$, we have:
\[m(Df^{-1}|_{E^s(x)})-1\geq C^s\max\bigg\{d(x ,\partial N)^{\beta^s},\bigg(\frac{d(f^{-1}(x), \partial N)}{d(x,\partial N)}\bigg)^{\gamma^s}-1\bigg\}\]
\textbf{Assumption R:} $\alpha>\frac{\beta}{\gamma}$, where $\gamma=\min\{\gamma^s ,\gamma^u\}>1,\beta=\max\{\beta^s ,\beta^u\}$.\\

The origin of the above assumptions can be referred to \cite{Wu}, \cite{SRB}.The assumptions  $U(i)$  and  $S(i)$  are stronger than the assumption  that “$E^{s}(x)$ , $E^{u}(x)$  are continuous”in \cite{Wu}, and  ensure that local stable and unstable manifolds can still be constructed on “pointwise pseudo orbits.” The assumptions  $U(ii)$  and  $S(ii)$  only make a slight modification to the exponents in the original assumptions $m(Df|_{E^u(x)})^{\kappa^u} - 1$  and  $m(Df^{-1}|_{E^s(x)})^{\kappa^s} - 1$, weaker than the original ones. In the literature \cite{Wu},  $\kappa^u$  and  $\kappa^s$  are only taken within  $(0,1)$ , but the condition  $\kappa^u, \kappa^s \in (0,1)$  is only used in Chapter 6 in \cite{Wu}, which involves Markov partitions and symbolic dynamics. Therefore, the conclusions in the first five chapters of \cite{Wu} still hold when  $\kappa^u = \kappa^s = 1$ . The assumptions $U(ii)$  and  $S(ii)$  ensure that the expansion rate along the unstable direction is greater than the rate of points move away from the boundary when points approach the boundary.
Assumption  $R$  is about the regularity of  $f$ . Thus, under the above assumptions, some conclusions in the literature \cite{Wu} still hold, we can still use the graph transform to construct local stable and unstable manifolds.

We now introduce two functions $\varepsilon(x)$ and $Q(x)$ that will be frequently used later, so that the discussion near the boundary depends on the distance of points to the boundary when  we perturb $f$ in the interior of $N$ .
$$ \varepsilon(x)=\varepsilon \min\{r_0^\beta, d(x, \partial N)^\beta\}$$
$$Q(x)=\varepsilon^\frac{2}{\alpha-\delta} \min\left\{r_0^\gamma,d(x,\partial N)^\gamma \right\}.$$
where $r_0=\min\{r_0^u,r_0^s\},$ $\delta$ is a fixed parameter satisfying $0<\delta<\min\{1,\alpha-\frac{\beta}{\gamma}\}$.In particular, when $\varepsilon$ is sufficiently small, it is easy to prove that\[Q(x)<d(x,\partial N)^\gamma<d(x,\partial N)\] \[Q(x)^{\alpha-\delta}=\varepsilon^2\min\bigg \{ r_0^   {(\alpha-\delta) \gamma },d(x,\partial N)^{(\alpha-\delta)\gamma}\bigg \} \leq \varepsilon \varepsilon(x).\]
\begin{lem}\label{扰动依据}
	For sufficiently small $\varepsilon>0$,we have \[\varepsilon\varepsilon(x)< \min\{m(Df|_{E^u(x)})-1,m(Df^{-1}|_{E^s(x)})-1,1-\|Df|_{E^{s}(x)}\|\} \]
\end{lem} 

\begin{proof}
	The discussion can be divided into the following cases:\\
	\textbf{Case 1:} If $d(x,\partial N)\geq r_0$, then $\varepsilon(x)=\varepsilon r_0^\beta$. Notice that $A=\{x\in N|d(x,\partial N)\geq r_0\}$ is compact, let$$a=\min_{x\in A}\{m(Df|_{E^u(x)})-1,m(Df^{-1}|_{E^s(x)})-1,1-\|Df|_{E^{s}(x)}\|\}.$$ When $\varepsilon<\sqrt{\frac{a}{r_0^\beta}}$, for all $ x\in N$ satisfying $d(x,\partial N)\geq r_0$, we have \[\varepsilon\varepsilon(x)< \min\{m(Df|_{E^u(x)})-1,m(Df^{-1}|_{E^s(x)})-1,1-\|Df|_{E^{s}(x)}\|\} .\]
	\textbf{Case 2:} If $d(x,\partial N)< r_0$, then $\varepsilon(x)=\varepsilon d(x,\partial N)^\beta$.Under assumption U(ii), S(ii), when $\varepsilon<1$, then$$\varepsilon\varepsilon(x)< \min \{m(Df|_{E^u(x)})-1,m(Df^{-1}|_{E^s(x)})-1\}.$$ Next we only need to show that when  $\varepsilon$  is sufficiently small,  $\varepsilon \varepsilon(x) < 1 - \|Df|_{E^{s}(x)}\|$  holds for all  $x \in N$ .
	We note that $$\|Df|_{E^{s}(x)}\|=(m(Df^{-1}|_{E^s(f(x))}))^{-1}.$$ Now we consider the following two cases:
	
	(1)When  $ d(f(x), \partial N) \geq r_0 $,  $ f(x) \in A $, then $$ a + 1 \leq m(Df^{-1}|_{E^s(f(x))}) = \|Df|_{E^{s}(x)}\|^{-1}.$$
	When  $ \varepsilon < \left(\frac{1 - (a + 1)^{-1}}{r_0^\beta}\right)^{\frac{1}{2}} $ , we have  $ \varepsilon \varepsilon(x) \leq \varepsilon^2 r_0^\beta < 1 - (a + 1)^{-1} \leq 1 - \|Df|_{E^{s}(x)}\| $ 
	
	(2)When $d(f(x),\partial N)< r_0$ ,by assumption S(ii), we have
	\begin{equation*}
		\begin{aligned}
			C^sd(f(x),\partial N)^{\beta^s}
			&\leq m(Df^{-1}|_{E^s(f(x))})-1\\
			&=\|Df|_{E^{s}(x)}\|^{-1}-1\\
			&=\|Df|_{E^{s}(x)}\|^{-1}(1-\|Df|_{E^{s}(x)}\|).
		\end{aligned}
	\end{equation*}
	Let $C^f=\max_{x\in M}\{\max\{\lVert D_xf\lVert,\lVert D_xf^{-1}\rVert\}\}$. Based on S(ii), let $b=(\frac{C^f-1}{C^s}+1)^{\frac{1}{\gamma^s}}\geq\frac{d(x,\partial N)}{d(f(x),\partial N)}$,
	take $\varepsilon<\min\{C^s,\frac{1}{b^\beta C^f}\}$,then for all $x\in N$,we have
	\begin{equation*}
		\begin{aligned}
			\varepsilon\varepsilon(x)
			&=\varepsilon^2d(x,\partial N)^\beta\\
			&\leq \varepsilon C^sb^\beta d(f(x),\partial N)^\beta\\
			&\leq\varepsilon C^sb^\beta d(f(x),\partial N)^{\beta^s}\\
			&\leq\varepsilon b^\beta\|Df|_{E^{s}(x)}\|^{-1}(1-\|Df|_{E^{s}(x)}\|)\\
			&\leq \varepsilon b^\beta m(Df^{-1}|_{E^s(f(x))})(1-\|Df|_{E^{s}(x)}\|)\\
			&\leq\varepsilon b^\beta C^f(1-\|Df|_{E^{s}(x)}\|)\\
			&\leq1-\|Df|_{E^{s}(x)}\rVert
		\end{aligned}
	\end{equation*}
	In summary, when $\varepsilon<\min\{\sqrt{\frac{a}{r_0^\beta}},1,\sqrt{\frac{1-(a+1)^{-1}}{r_0^\beta}},C^s,\frac{1}{b^\beta C^f}\}$, both cases 1 and 2 hold, that is to say,\[\varepsilon\varepsilon(x)\leq \min\{m(Df|_{E^u(x)})-1,m(Df^{-1}|_{E^s(x)})-1,1-\|Df|_{E^{s}(x)}\|\} .\]
\end{proof}
\begin{cor}
	There exists a constant $C_0$ such that for sufficiently small  $\varepsilon>0$,we have$$C_0\varepsilon(x)< \min\{m(Df|_{E^u(x)})-1,m(Df^{-1}|_{E^s(x)})-1,1-\|Df|_{E^{s}(x)}\|\} .$$
\end{cor}
This can be  directly obtained from the proof  of Lemma \ref{扰动依据}.

\section{Admissible manifolds and graph transform}
For all $ x\in N$, let $B^{s}(Q(x))$ and $B^{u}(Q(x))$  be the closed balls centered at the origin with radius $Q(x)$ on $E^{s}(x)$ and $E^{u}(x)$ respectively. Denote $S_{x}=B^{u}(Q(x))\oplus B^{s}(Q(x))$, which is a $Q(x)$- square centered at $O_{x}$.
Recall the fact about the exponential map. Since $M$ is compact, there exists $\rho > 0$ such that for all $ x\in M$ and $v\in T_xM(\rho)$, $d(\exp_x(v),x)=\|v\|$. When $\varepsilon$ is small enough so that $Q(x)<\frac{\rho}{\sqrt{2}C^f}$, $f\circ \exp_x$ maps $S_x$ into $B(f(x),\sqrt{2}C^fQ(x))$, where $C^f=\max_{x\in M}\{\max\{\lVert D_xf\rVert,\lVert D_xf^{-1}\rVert\}\}$.

In this case, for $y\in N$ satisfying $d(f(x),y)< \rho-\sqrt{2}C^fQ(x)$,
\[F_{xy}=\exp_y^{-1}\circ f\circ \exp_x:S_x\to T_yM \]
is well defined and is a $C^{1 + \alpha}$ diffeomorphism onto its image.
$F_{xy}$ is equivalent to a map defined between two Euclidean spaces, and its differentiability can be considered:
\[D_0F_{xy}=
\begin{pmatrix}
	D_{xy}^{uu}&D_{xy}^{us}  \\
	D_{xy}^{su}&D_{xy}^{ss}  \\	
\end{pmatrix}\]
where $D_{xy}^{us}=P_y^u\circ D_{f(x)}\exp_y^{-1}\circ D_xf|_{E^s(x)}:B_{x}^{s}(Q(x))\to E^u(y)$, and $P_y^u$ is the projection in the direction of $E^u(x)$.

\begin{prop}\cite{Wu}\label{相差估计}
	For all sufficiently small $\varepsilon>0$, there exists a function $\delta^u:N\to \mathbb{R}_{+}$ such that for  all $ x\in N$, $\delta^u(x)<\min\{\varepsilon\varepsilon(x),\rho - \sqrt{2}C^fQ(x)\}$, and when $y$ satisfies $d(f(x),y)<\delta^u(x)$, we have the following results for $F_{xy}|_{S_x}=D_0F_{xy}+\varphi_{xy}$ under the decomposition $E^s\oplus E^u$:
	\[\lVert D_{xy}^{us}\rVert\leq\varepsilon\varepsilon(x),\quad\left|m(D_{xy}^{uu})-m(Df|_{E^{u}(x)})\right|\leq\varepsilon\varepsilon(x),\]	\[\lVert D_{xy}^{su}\rVert\leq\varepsilon\varepsilon(x),\quad\left|\|D_{xy}^{ss}-\|Df|_{E^{s}(x)}\|\right|\leq\varepsilon\varepsilon(x).\]
	In addition, $\|\ D\varphi_{xy}\|_{C^0}\leq\varepsilon\varepsilon(x)$. For all $v, w\in S_x$, $\|D_v\varphi_{xy}-D_w\varphi_{xy}\|\leq\varepsilon\varepsilon(x)\|v - w\|^{\frac{\delta}{2}}$ holds.
	For $F_{xy}^{-1}=\exp_x^{-1}\circ f^{-1}\circ \exp_y$, when $d(f^{-1}(y),x)<\delta^s(y)$, similar conclusions hold.
\end{prop}

This is exactly Proposition 3.1 in \cite{Wu}. In the original paper, $\delta^u(x)$ is a locally monotonic function, which is required for the  study of Markov partitions and symbolic systems.  Here, we have made some simple modifications and the proposition still holds and does not affect the definition of the graph transform and the construction of local stable and unstable manifolds below.
\begin{Def}\cite{Wu} For $\forall x\in N$, a $u$-manifold and an $s$-manifold at $x$ refer to
	\[W_x^u=\exp_x\{(v^u,\phi_x^u(v^u))\in T_xM:\| v^u\|\leq Q(x)\}\subset M,\]
	\[W_x^s=\exp_x\{(\phi_x^s(v^s),v^s)\in T_xM:\| v^s\|\leq Q(x)\}\subset M,\]
	respectively, where $\phi^u_x:B^u(Q(x))\to E^s(x)$ and $\phi^s_x:B^s(Q(x))\to E^u(x)$ are $C^{1+\frac{\delta}{2}}$ functions such that $\|\phi_x^t\|_{C^0} \leq Q(x)$ ($t = s/u$). Furthermore, a $t$-manifold satisfying the following conditions is called a $t$-admissible manifold:
	\[\|\phi_x^t(0)\|\leq10^{-3}Q(x),\quad \|D\phi_x^t \|_{C^0}\leq \sqrt {\varepsilon},\quad \|D\phi_x^t \| _{C^{\frac {\delta}{2}}}=\|D\phi_x^t\|_{C^0} +Hol_{\frac{\delta}{2}}(D\phi_x^t)\leq\frac{1}{2},\]
	where $Hol_{\frac{\delta}{2}}(D\phi_x^t)=\sup_{v\neq w}\frac{\|D_v\phi_x^t-D_w\phi_x^t\|}{\Arrowvert v-w\|^\frac{\delta}{2}}$.
\end{Def}

\begin{Def}\cite{Wu}
	Denote $\mathscr{W}_x^t$ as the set of all $t$-admissible manifolds at $x$. When $d(f(x),y)<\delta^u(x)Q(x)^2Q(f(x))$ and $d(f^{-1}y,x)<\delta^s(y)Q(y)^2Q(f^{-1}y)$, for $F_{xy}|_{S_x}=D_0F_{xy}+\varphi_{xy}$,  denote $F_{xy}(v_x^u,\phi_x^u(v_x^u))=(G_{xy}^{\phi_x^u}(v_x^u),L_{xy}^{\phi_x^u}(v_x^u))$.\\
	(1) The unstable graph transform refers to the map $\mathscr{T}_{xy}^u:\mathscr{W}_x^u\to\mathscr{W}_y^u$ which sends $W_x^u\in\mathscr{W}_x^u$ with the representing function $\phi_x^u$ to $\mathscr{T}_{xy}^u({W}_x^u)\in\mathscr{W}_y^u$ with the representing function $\phi_y^u := L_{xy}^{\phi_x^u}(G_{xy}^{\phi_x^u})^{-1}$.That is,
	\[\mathscr{T}_{xy}^u[{W}_x^u]=\exp_y\{(v^u,\phi_y^u(v^u)):\|v^u\|\leq Q(y)\}\subset f(W_x^u)\]
	(2)The stable graph transform refers to the map $\mathscr{T}_{xy}^s:\mathscr{W}_y^s\to\mathscr{W}_x^s$ which maps ${W}_y^s\in\mathscr{W}_y^s$ to $\mathscr{T}_{xy}^s[{W}_y^s]\in\mathscr{W}_x^s$ with the representing function $\phi_x^s$. That is,
	\[
	\mathscr{T}_{xy}^s[{W}_y^s]=\exp_x\{(\phi_x^s(v^s),v^s)):\|v^s\|\leq Q(y)\}\subset f^{-1}(W_y^s)
	\]
\end{Def}

\begin{lem}
	For all sufficiently small $\varepsilon > 0$, if $d(f(x), y) < \delta^u(x)Q(x)^2Q(f(x))$ and $d(f^{-1}(y), x) < \delta^s(y)Q(y)^2Q(f^{-1}(y))$, then for any two $u$-admissible manifolds $W_1$ and $W_2$ at $x$, the following inequality holds:
	$$\|\widetilde{\phi}_1 - \widetilde{\phi}_2\|\leq\|Df|_{E^s(x)}\|^{\frac{1}{2}}\|\phi_1 - \phi_2\|$$
	where $\phi_1$ and $\phi_2$ are the representing functions of $W_1$ and $W_2$ respectively, $\widetilde{\phi}_1$ and $\widetilde{\phi}_2$ are the representing functions of $\mathscr{T}_{xy}^u({W}_1)$ and $\mathscr{T}_{xy}^u({W}_2)$ respectively, and $\widetilde{\phi}_i = L_{xy}^{\phi_i}(G_{xy}^{\phi_i})^{-1}$ for $i = 1, 2$.
\end{lem}

\begin{proof}
	For $\forall v_y\in B_y^u(Q(y))$, let $v_x^i=(G_{xy}^{\phi_i})^{-1}(v_y)$ for $i = 1,2$. Then
	\begin{align*}
		\Arrowvert\widetilde{\phi}_1(v_y)-\widetilde{\phi}_2(v_y)\Arrowvert
		&=\Arrowvert \widetilde{\phi}_1\circ G_{xy}^{\phi_1}(v_x^1)-\widetilde{\phi}_2\circ G_{xy}^{\phi_1}(v_x^1)\Arrowvert\\
		&\leq\Arrowvert \widetilde{\phi}_1\circ G_{xy}^{\phi_1}(v_x^1)-\widetilde{\phi}_2\circ G_{xy}^{\phi_2}(v_x^1)\Arrowvert+\Arrowvert\widetilde{\phi}_2\circ G_{xy}^{\phi_2}(v_x^1)-\widetilde{\phi}_2\circ G_{xy}^{\phi_1}(v_x^1)\Arrowvert\\
		&\leq\|D^{ss}_{x,y}(\phi_1(v_x^1)-\phi_2(v_x^1))\|+\|\varphi^{s}_{x,y}(v_x^1,\phi_1(v_x^1))-\varphi^{s}_{x,y}(v_x^1,\phi_2(v_x^1))\|\\&\quad+Lip(\widetilde{\phi}_2)\Arrowvert G_{xy}^{\phi_2}(v_x^1)-G_{xy}^{\phi_1}(v_x^1)\Arrowvert\\
		&\leq\|D^{ss}_{x,y}\|\|(\phi_1(v_x^1)-\phi_2(v_x^1))\|+Lip(\varphi_{x,y})\|\phi_1(v_x^1)-\phi_2(v_x^1)\|+ \\&\quad Lip(\widetilde{\phi}_2)(D_{x,y}^{us}+Lip(\varphi_{x,y})) \Arrowvert\phi_2(v_x^1)-\phi_1(v_x^1)\Arrowvert\\
		&\leq(\|D^{ss}_{x,y}\|+\|D^{us}_{x,y}\|+2Lip(\varphi_{x,y}))\|\phi_1(v_x^1)-\phi_2(v_x^1)\|\\
		&\leq(Df|_{E^s(x)}+4\varepsilon\varepsilon(x))\Arrowvert\phi_1(v_x^1)-\phi_2(v_x^1)\Arrowvert
	\end{align*}
	Let $a_0 = \inf_{x\in N}\frac{\sqrt{\|Df|_{E^s(x)}\|}-\|Df|_{E^s(x)}\|}{1 - \|Df|_{E^s(x)}\|}$.
	By the corollary of Lemma \ref{扰动依据},  for sufficiently small $\varepsilon$ such that $\varepsilon<\frac{a_0C_0}{4}$, we have $$\|\widetilde{\phi}_1-\widetilde{\phi}_2\|
	\leq(\|Df|_{E^s(x)}\|+a_0C_0\varepsilon(x))\|\phi_1-\phi_2\|\leq\|Df|_{E^s(x)}\|^{\frac{1}{2}}\|\phi_1-\phi_2\|.$$ 
\end{proof}

\section{Stability of pointwise hyperbolic systems}
In this part, we first introduce the definition of pointwise pseudo orbits. Based on the graph transform and related lemmas in the preliminaries, we construct local stable and unstable manifolds on pointwise pseudo orbits. Consequently, we obtain the shadowing lemma and pointwise expansivity on pointwise pseudo orbits. On this basis, we further explore the preservation of pointwise hyperbolicity under sufficiently small perturbations, demonstrating that the perturbed system still has certain hyperbolicity under the assumption K. According to these conclusions, we finally illustrate the stability of pointwise hyperbolic systems under sufficiently small perturbations.
\subsection{Pointwise pseudo orbit shadowing and expansivity}
\begin{Def}\label{pointwise_pseudo_orbit}
	We call $\{x_n\}_{n\in \mathbb{Z}}\subset N$ an $\varepsilon$-pointwise pseudo orbit of $f$ if $\{x_n\}_{n\in \mathbb{Z}}$ satisfies:
	
	(1) For any $n > 0$, 
	\[d(f(x_n),x_{n+1})<\delta^u(x_n)Q(x_n)^2Q(f(x_n))\]
	\[d(f^{-1}(x_{n+1}),x_n)<\delta^s(x_{n+1})Q(x_{n+1})^2Q(f^{-1}(x_{n+1}))\]
	
	(2)$$\prod_{i=0}^\infty\|Df|_{E^s(x_{-i})}\|=0,\quad 
	\prod_{i=0}^\infty\|Df^{-1}|_{E^u(x_i)}\|=0.$$
	
\end{Def}

Notice that for $\forall x\in N$, the true orbit $\{f^nx\}_{n\in \mathbb{Z}}$ of $f$ is also a pointwise pseudo orbit of $f$. In fact, according to the initial assumptions U(i) and S(i), we have 
$$\prod_{i = 0}^{\infty}m(Df^{-1}|_{E^s(f^{-i}x)})=\infty,\quad\prod_{i = 0}^{\infty}m(Df|_{E^u(f^ix)})=\infty,\quad \forall x\in N$$
From this, we can infer that
$$\prod_{i = 0}^{\infty}\|Df|_{E^s(f^{-i}x)}\| = 0,\quad\prod_{i = 0}^{\infty}\|Df^{-1}|_{E^u(f^ix)}\| = 0.$$
Therefore, for $\forall x\in N$, the true orbit $\{f^nx\}_{n\in \mathbb{Z}}$ of $f$ satisfies the definition of a pointwise pseudo orbit.

Let $\{x_n\}_{n\in \mathbb{Z}}$ be an $\varepsilon$-pointwise pseudo orbit in $N$. Choose an arbitrary $u$-admissible manifold $W_{-n}^u$ at $x_{-n}$ and an arbitrary $s$-admissible manifold $W_{n}^s$ at $x_{n}$. Denote $\mathscr{T}_{-n,-n + 1}^u$ as the unstable graph transform from $x_{-n}$ to $x_{-n+1}$, which maps the $u$-admissible manifold at $x_{-n}$ to the $u$-admissible manifold at $x_{-n+1}$. Denote $\mathscr{T}_{n,n + 1}^s$ as the stable graph transform from $x_{n+1}$ to $x_{n}$, which maps the $s$-admissible manifold at $x_{n+1}$ to the $s$-admissible manifold at $x_{n}$. Then we have the following proposition:
\begin{prop}\label{稳定流形}\cite{Wu}
	For any sufficiently small $\varepsilon > 0$,
	
	(1)\hfill $W^u(\{x_{-n}\}_{n\geq0}):=\lim\limits_{n\to\infty}\mathscr{T}_{-1,0}^u\circ\cdots\circ \mathscr{T}_{-n,-n+1}^u(W_{-n}^u)$\hfill\mbox{}
	$$W^s(\{x_{n}\}_{n\geq 0}):=\lim_{n\to\infty}\mathscr{T}_{0,1}^s\circ\cdots\circ \mathscr{T}_{n-1,n}^s(W_{n}^s)$$
	are the $u$-admissible manifold and the $s$-admissible manifold at $x_0$ respectively, and we call them the local unstable and stable manifolds of the pointwise pseudo orbit $\{x_n\}_{n\in \mathbb{Z}}$.
	
	(2)\hfill 
	$W^u(\{x_{-n}\}_{n\geq0})=\{x\in \exp_{x_0}(S_{x_0}):\forall k\geq0,f^{-k}x \in \exp_{x_{-k}}(S_{x_{-k}})\}$\hfill\mbox{}
	\[W^s(\{x_{n}\}_{n\geq 0})=\{x\in \exp_{x_0}(S_{x_0}):\forall k\geq0,f^{k}x \in \exp_{x_{k}}(S_{x_{k}})\}\]
\end{prop}

The proof process of Proposition \ref{稳定流形} is basically the same as that of Proposition 5.2 in \cite{Wu}. The only difference is that in \cite{Wu}, the local stable and unstable manifolds are constructed on the $\varepsilon$-recurrent pointwise pseudo orbit. Taking the construction of the local unstable manifold as an example, the recurrence property ensures the existence of a constant subsequence. Then $\prod_{i = 0}^{\infty}\|Df|_{E^s(x_{-i})}\|^{\frac{1}{2}} = 0$. Consequently, $W^u(\{x_{-n}\}_{n\geq0})$ is a $u$-admissible manifold at $x_0$. In this paper, the local unstable manifold is constructed on the $\varepsilon$-pointwise pseudo orbit. Although the recurrence property is lacking, according to the definition of the pointwise pseudo orbit, we still have $\prod_{i = 0}^{\infty}\|Df|_{E^s(x_{-i})}\|^{\frac{1}{2}} = 0$. Therefore, the same conclusion holds. To ensure the integrity of this paper, the detailed proof process is still presented as follows. Only the case of the local unstable manifold is given below, and the case of the local stable manifold is similar.

\begin{proof}
	(1) Generally, for the $u$-admissible manifolds $W_1$ and $W_2$ at $x$, we denote $\|W_1 - W_2\|=\|\phi_1-\phi_2\|$, where $\phi_1$ and $\phi_2$ are their representing functions.
	For $n\geq0$, arbitrarily choose $W_{-n}^u\in\mathscr{W}_x^u$ at $x_{n}$.
	For $m > n$, we have
	\begin{equation*}
		\begin{aligned}
			&\|\mathscr{T}^{u}_{-1,0}\circ\cdots\circ\mathscr{T}^{u}_{-n,-n+1}(W_{-n}^u)-\mathscr{T}^{u}_{-1,0}\circ\cdots\circ\mathscr{T}^{u}_{-m,-m+1}(W_{-m}^u)\|\\
			&\leq\prod^n_{i=1}\|Df|_{E^s(x_{-i})}\|^{\frac{1}{2}}\|\|W_{-n}^u-\mathscr{T}^{u}_{-n-1,-n}\circ\cdots\circ\mathscr{T}^{u}_{-m,-m+1}(W_{-m}^u)\|\\
			&\leq\prod^n_{i=1}\|Df|_{E^s(x_{-i})}\|^{\frac{1}{2}}\rightarrow0(n\rightarrow\infty)
		\end{aligned}
	\end{equation*}
	Therefore, the sequence $\{\mathscr{T}^{u}_{-1,0}\circ\cdots\circ\mathscr{T}^{u}_{-n,-n + 1}(W_{-n}^u)\}_{n\geq1}$ converges. Next, we  will show that the limit is a $u$-admissible manifold at $x_0$.
	
	Since $\mathscr{T}^{u}_{-1,0}\circ\cdots\circ\mathscr{T}^{u}_{-m,-m + 1}(W_{-m}^u)\to W^u(\{x_n\}_{n\geq0})$ when $n\to\infty$, its representing functions $\phi^{(n)}\to\phi$ uniformly on $B_{x_0}^u(Q(x_0))$. From $\|\ D\phi^{(n)}\|_{C^{\frac {\delta}{2}}}=\|\ D\phi^{(n)}\|_{C^0}+\text{Hol}_{\frac{\delta}{2}}(D\phi^{(n)})\leq\frac{1}{2}$, we know that there exists a subsequence $D\phi^{(n_k)}\xrightarrow{k\rightarrow\infty}\sigma$ uniformly and $\|\sigma\|_{C^{\delta/2}}\leq\frac{1}{2}$. For $\forall v\in B^u_{x_0}(Q(x_0))$, we have
	\[
	\phi^{(n_k)}(v)=\phi^{(n_k)}(0)+\int_0^1D_{\lambda v}\phi^{(n_k)}vd\lambda\xrightarrow{k\rightarrow\infty}\phi(0)+\int_0^1\sigma(\lambda v)vd\lambda
	.\]
	So $\phi$ is differentiable and $D\phi = \sigma$. Since $\{D\phi^{(n)}\}$ has only one limit point, we obtain that $D\phi^{(n)}\to D\phi$ uniformly. Thus, we have $\|\phi(0)\|\leq10^{-3}Q(x_0)$, $\|\ D\phi\|_{C^0}\leq \sqrt {\varepsilon}$, $\|\ D\phi\|_{C^{\frac {\delta}{2}}}\leq\frac{1}{2}$, which means that $W^u(\{x_{-n}\}_{n\geq0})$ is a $u$-admissible manifold at $x_0$.
	
	(2) First, we can prove the inclusion “$\subset$”: Since $W^u(\{x_{-n}\}_{n\geq0})$ is a $u$-manifold at $x_0$, by the definition, for $\forall x\in W^u(\{x_{-n}\}_{n\geq0})$, we have $x\in \exp_{x_0}(S_{x_0})$. According to the definition of the graph transform, 
	\begin{equation*}
		\begin{aligned}
			W^u(\{x_{-n}\}_{n\geq0})&=\lim_{n\to\infty}\mathscr{T}_{-1,0}^u\circ\cdots\circ \mathscr{T}_{-n,-n+1}^u(W_{-n}^u)\\
			&=\mathscr{T}_{-1,0}^u[\lim_{n\to\infty}\mathscr{T}_{-2,-1}^u\circ\cdots\circ \mathscr{T}_{-n,-n+1}^u(W_{-n}^u)]\\
			&=\mathscr{T}_{-1,0}^u[W^u(\{x_{-n-1}\}_{n\geq0})]\\
			&\subset f(W^u(\{x_{-n-1}\}_{n\geq0}))
		\end{aligned}
	\end{equation*}
	thus, $f^{-k}(x)\in f^{-k}[W^u(\{x_{-n}\}_{n\geq0})]\subset W^u(\{x_{-n-k}\}_{n\geq0})\subset\exp_{x_{-k}}(S_{x_{-k}})$.
	
	Secondly, we prove the inclusion “$\supset$”: For $x\in N$ such that $\forall k\geq0$, $$f^{-k}x\in\exp_{x_{-k}}(S_{x_{-k}}).$$ Without loss of generality, denote $f^{-k}x = \exp_{x_{-k}}(v^u_{-k},v^s_{-k})$.
	Let the representing function of $W^u(\{x_{-n}\}_{n\geq0})$ be $\phi^u$, and then $\exp_{x_0}(v^u_0,\phi^u(v^u_0))\in W^u(\{x_{-n}\}_{n\geq0})$.
	Next, we only need to show that $\phi^u(v^u_0)=v^s_0$.
	According to the previous analysis, for $k\geq0$, let
	\[f^{-k}\exp_{x_0}(v^u_0,\phi^u(v^u_0))=\exp_{x_{-k}}(\widetilde{v}^u_{-k},\widetilde{v}^s_{-k})\in W^u(\{x_{-n - k}\}_{n\geq0}).\]
	Now we consider $F^{-1}_{x_{-k-1},x_{-k}}=\exp_{x_{-k-1}}^{-1}\circ f^{-1}\circ\exp_{x_{-k}}$. For $(v^u_{-k},v^s_{-k})\in S_{x_{-k}}$, it has the following form:\\
	$F^{-1}_{x_{-k-1},x_{-k}}(v^u_{-k},v^s_{-k})=( \widetilde{D}^{uu}_{-k-1,_-k}(v^u_{-k})+\widetilde{D}^{us}_{-k-1,_-k}(v^s_{-k})+\widetilde{\varphi}^{u}_{-k-1,_-k}(v^u_{-k},v^s_{-k}), \widetilde{D}^{ss}_{-k-1,_-k}\\(v^s_{-k})+\widetilde{D}^{su}_{-k-1,_-k}(v^u_{-k})+\widetilde{\varphi}^{s}_{-k-1,_-k}(v^u_{-k},v^s_{-k})) $,
	where $$\widetilde{D}_{-k-1,_-k}^{us}=P_{x_{-k-1}}^u\circ D_{f^{-1}(x_{-k})}\exp_{x_{-k-1}}^{-1}\circ D_{x_{-k}}f^{-1}|_{E^s(x_{-k})}:B_{x_{-k}}^{s}(Q(x_{-k}))\to  E^u(x_{-k-1}),$$$P_{x_{-k - 1}}^u$ is the projection in the direction of $E^u(x_{-k - 1})$. According to the conclusion in Proposition \ref{相差估计}, we know that
	\begin{equation*}
		\begin{aligned}
			\lVert \widetilde {v}^u_{-k-1}-v^u_{-k-1}\lVert
			&=
			\lVert	F^{-1}_{-k-1,-k}(\widetilde {v}^u_{-k},\widetilde {v}^s_{-k})|_{E^u}-F^{-1}_{-k-1,-k}(v^u_{-k},v^s_{-k})|_{E^u}\lVert\\
			&\leq\|\widetilde{D}^{uu}_{-k-1,-k}\| \|\widetilde{v}^u_{-k}-v^u_{-k}\|+\|\widetilde{D}^{us}_{-k-1,-k}\|\|\widetilde{v}^s_{-k}-v^s_{-k}\|+\\ &\quad Lip(\widetilde{\varphi}^{u}_{-k-1,-k})(\|\widetilde{v}^u_{-k}-v^u_{-k}\|+\|\widetilde{v}^s_{-k}-v^s_{-k}\|)\\
			&\leq(\|Df^{-1}|_{E^u(x_{-k})}\|+2\varepsilon\varepsilon(x_{-k}))\|\widetilde{v}^u_{-k}-v^u_{-k}\|+2\varepsilon\varepsilon(x_{-k})\|\widetilde{v}^s_{-k}-v^s_{-k}\|.
		\end{aligned}
	\end{equation*}
	\begin{equation*}
		\begin{aligned}
			\lVert \widetilde {v}^s_{-k-1,-k}\lVert
			&=
			\lVert	F^{-1}_{-k-1,-k}(\widetilde {v}^u_{-k},\widetilde {v}^s_{-k})|_{E^s}-F^{-1}_{-k-1,-k}(v^u_{-k},v^s_{-k})|_{E^s}\lVert\\
			&\geq m(\widetilde{D}^{ss}_{-k-1,-k}) \|\widetilde{v}^s_{-k}-v^s_{-k}\|-m(\widetilde{D}^{us}_{-k-1,-k})\|\widetilde{v}^u_{-k}-v^u_{-k}\|-\\ &\quad Lip(\widetilde{\varphi}^{s}_{-k-1,-k})(\|\widetilde{v}^u_{-k}-v^u_{-k}\|+\|\widetilde{v}^s_{-k}-v^s_{-k}\|)\\
			&\geq(m(Df^{-1}|_{E^s(x_{-k})})-2\varepsilon\varepsilon(x_{-k}))\|\widetilde{v}^s_{-k}-v^s_{-k}\|-2\varepsilon\varepsilon(x_{-k})\|\widetilde{v}^u_{-k}-v^u_{-k}\|
		\end{aligned}
	\end{equation*}
	
	Next, we will prove by mathematical induction that for $\forall k\geq0$, $\|\widetilde{v}^u_{-k}-v^u_{-k}\|\leq\|\widetilde{v}^s_{-k}-v^s_{-k}\|\leq\|\widetilde{v}^s_{-k - 1}-v^s_{-k - 1}\|$.
	
	When $k = 0$, $\|\widetilde{v}^u_{-k}-v^u_{-k}\| = 0$. And according to the corollary of Lemma \ref{扰动依据}, when $\varepsilon$ is small enough, $m(Df^{-1}|_{E^s(x_{-k})})-2\varepsilon\varepsilon(x_{-k})\geq1$. So the conclusion holds.
	Assume that the conclusion holds when $k = m$. Now we consider the case when $k=m + 1$.
	According to the corollary of Lemma \ref{扰动依据}, for $\varepsilon\leq C_0/4$,
	\begin{equation*}
		\begin{aligned}
			\lVert \widetilde {v}^u_{-m-1}-v^u_{-m-1}\lVert
			&\leq(\|Df^{-1}|_{E^u(x_{-m})}\|+2\varepsilon\varepsilon(x_{-m}))\|\widetilde{v}^u_{-m}-v^u_{-m}\|+2\varepsilon\varepsilon(x_{-m})\|\widetilde{v}^s_{-m}-v^s_{-m}\|\\
			&\leq(\|Df^{-1}|_{E^u(x_{-m})}\|+4\varepsilon\varepsilon(x_{-m}))\|\|\widetilde{v}^s_{-m}-v^s_{-m}\| \\
			&\leq\|\widetilde{v}^s_{-m}-v^s_{-m}\|\\
			&\leq\|\widetilde{v}^s_{-m-1}-v^s_{-m-1}\|.
		\end{aligned}
	\end{equation*}
	\begin{equation*}
		\begin{aligned}
			\lVert \widetilde {v}^s_{-m-2}-v^s_{-m-2}\lVert
			&\geq(m(Df^{-1}|_{E^u(x_{-m})})-2\varepsilon\varepsilon(x_{-m-1}))\|\widetilde{v}^s_{-m-1}-v^s_{-m-1}\|-2\varepsilon\varepsilon(x_{-m-1})\
			\\&\quad|\widetilde{v}^u_{-m-1}-v^u_{-m-1}\|\\
			&\geq(m(Df^{-1}|_{E^u(x_{-m})})-4\varepsilon\varepsilon(x_{-m-1}))\|\widetilde{v}^s_{-m-1}-v^s_{-m-1}\| \\
			&\geq\|\widetilde{v}^s_{-m-1}-v^s_{-m-1}\|.
		\end{aligned}
	\end{equation*}
	Therefore, the conclusion also holds when $k = m+ 1$.
	
	Let $a_1=\inf_{x\in N}\frac{m\left(Df^{-1}\big|_{E^s(x)}\right)-\sqrt{m\left(Df^{-1}\big|_{E^s(x)}\right)}}{m\left(Df^{-1}\big|_{E^s(x)}\right)-1}$,
	According to the assumption S(ii), for all $k\geq0$, when $\varepsilon\leq\frac{a_1C_0}{4}$,
	\begin{equation*}
		\begin{aligned}
			\lVert \widetilde {v}^s_{-k-2}-v^s_{-k-2}\lVert
			&\geq(m(Df^{-1}|_{E^u(x_{-m})})-a_1C_0\varepsilon(x_{-k-1}))\|\widetilde{v}^s_{-k-1}-v^s_{-k-1}\| \\
			&\geq\sqrt{m(Df^{-1}|_{E^s(x)}}\|\widetilde{v}^s_{-k-1}-v^s_{-k-1}\|.
		\end{aligned}
	\end{equation*}
	Then $$\|\widetilde{v}^s_{-k-1}-v^s_{-k-1}\|\geq\prod_{i=0}^{k-1}m(Df^{-1}|_{E^s(x_{-i})})^{\frac{1}{2}}\|\widetilde{v}^s_{0}-v^s_{0}\|=\prod_{i=0}^{k-1}m(Df^{-1}|_{E^s(x_{-i})})^{\frac{1}{2}}\|\phi^u(v^u_0)-v^s_{0}\|.$$
	Since $$\|\widetilde{v}^s_{-k-1}-v^s_{-k-1}\|\leq2Q(x_{-k})<2,  \prod_{i=0}^{k-1}m(Df^{-1}|_{E^s(x_{-i})})^{\frac{1}{2}}\xrightarrow{k\rightarrow\infty}\infty,$$
	we have $\|\phi^u(v^u_0)-v^s_{0}\|=0$, that is    $\phi^u(v^u_0)=v^s_0 $.
\end{proof}
\vspace{0.5cm}
\begin{Def}
	An $\varepsilon$-pointwise pseudo orbit $\{x_n\}_{n\in\mathbb{Z}}\subset N$ is said to be shadowed by $x\in N$ if for all $n\in\mathbb{Z}$, $f^nx\in\exp_{x_n}(S_{x_n})$.
\end{Def}

\begin{lem}\textbf{(Shadowing lemma)}\label{跟踪引理}
	If $f:M\rightarrow M$ is pointwise hyperbolic on the invariant set $N$ and satisfies the assumptions U, S, R, then for all sufficiently small $\varepsilon>0$, every $\varepsilon$-pointwise hyperbolic pseudo orbit $\{x_n\}_{n\in\mathbb{Z}}$ in $N$ can be shadowed by a unique $x\in N$, and $d(x,x_0)\leq\frac{1}{50}Q(x_0)$.
\end{lem}
\begin{proof}
	By Proposition \ref{稳定流形}, $V^u(\{x_{-n}\}_{n\geq0})$ and $V^s(\{x_{n}\}_{n\geq0})$ are the $u$-admissible manifold and $s$-admissible manifold at $x_0$ respectively. According to the definition of pseudo orbit shadowing and Proposition \ref{稳定流形}(2), in fact, we only need to prove that there exists $x\in N$ such that $V^u(\{x_{-n}\}_{n\geq0})\cap V^s(\{x_{n}\}_{n\geq 0})=\{x\}$.
	
	Let $\phi^u$ and $\phi^s$ be the representing functions of $V^u(\{x_{-n}\}_{n\geq0})$ and $V^s(\{x_{n}\}_{n\geq0})$ respectively. First, it can be shown that $V^u(\{x_{-n}\}_{n\geq0})$ and $V^s(\{x_{n}\}_{n\geq 0})$ have an intersection point, that is, there exists $ x = \exp_{x_0}(v^u,v^s)$ such that $v^u=\phi^s(v^s)$ and $v^s=\phi^u(v^u)$ hold simultaneously.
	This can be transformed into finding the fixed point of $\phi^s\circ\phi^u$ on $B^{u}_{x_0}(Q(x_0))$. Notice that when $0 < \sqrt{\varepsilon}<10^{-2}-10^{-3}$, for $\forall v\in B^{t}_{x_0}(Q(x_0))(t = u/s)$ 
	,$$\Arrowvert\phi^t(v)\Arrowvert\leq\Arrowvert\phi^t(0)\Arrowvert+Lip(\phi^t)\|v\| \leq10^{-3}Q(x_0)+\sqrt{\varepsilon}\|v\|\leq10^{-2}Q(x_0)<Q(x_0),$$
	That is to say, $\phi^s\circ\phi^u$ maps $B^{u}_{x_0}(Q(x_0))$ into itself.
	Moreover, since $\|D(\phi^s\circ\phi^u)\|\leq\|D\phi^s\|\|D\phi^u\|\leq\varepsilon < 1$, $\phi^s\circ\phi^u$ is a contraction mapping. Then it has a unique fixed point $v^u_0$ in $B^{u}_{x_0}(Q(x_0))$ such that $\phi^s\circ\phi^u(v^u_0)=v^u_0$. Denote $w=(v^u_0,\phi^u(v^u_0))$, then $V^u(\{x_{-n}\}_{n\geq0})$ and $V^s(\{x_{n}\}_{n\geq0})$ intersect at the point $x = \exp_{x_0}(w)$.
	
	In fact, $\phi^s\circ\phi^u$ maps $B^{u}_{x_0}(10^{-2}Q(x_0))$ into itself, and $v^u_0$ is also its unique fixed point in $B^{u}_{x_0}(10^{-2}Q(x_0))$. Thus,
	$$\|w\|\leq \|v^u_0\|+\|v^s_0\|\leq10^{-2}Q(x_0)+10^{-2}Q(x_0)=\dfrac{1}{50}Q(x_0).$$
	Furthermore, $d(x,x_0)=\|w\|\leq\frac{1}{50}Q(x_0)<\frac{1}{50}d(x_0,\partial{N})$, then $x\in N$.
	According to the definition of pointwise pseudo orbit shadowing, $x$ shadows the pointwise pseudo orbit $\{x_n\}_{n \in\mathbb{Z}}$.
	If there is another point $y$ that shadows $\{x_n\}_{n\in \mathbb{Z}}$, then $y\in V^u(\{x_{-n}\}_{n\geq0})\cap V^s(\{x_{n}\}_{n\geq 0})$. Due to the uniqueness of the intersection point of $V^u(\{x_{-n}\}_{n\geq0})$ and $V^s(\{x_{n}\}_{n\geq 0})$, we can immediately get $x = y$. That is to say, the pointwise hyperbolic pseudo orbit $\{x_n\}_{n\in \mathbb{Z}}$ is shadowed by a unique $x\in N$.
\end{proof}
\vspace{0.4cm}
\begin{Def}
	A $C^{1 + \alpha}$ diffeomorphism $f:M\to M$ is said to be pointwise expansive on the invariant set $N$ if for every $x\in N$, there exists a sequence of functions $\{r_n(x)\}_{n\in \mathbb{Z}}$ such that when $d(f^nx,f^ny)\leq r_n(x)$ for every $n\in \mathbb{Z}$, then $x = y$.
\end{Def}

\begin{them}\label{逐点可扩}
	If the $C^{1 + \alpha}$ diffeomorphism $f:M\rightarrow M$ is pointwise hyperbolic on the invariant set $N$ and satisfies the  assumptions U, S, R, then $f$ is pointwise expansive on $N$.
\end{them}
\begin{proof}
	It can be seen from Proposition \ref{稳定流形} that for any sufficiently small $\varepsilon>0$, for any given $x\in N$, we can take the sequence of functions $r_n(x) = Q(f^nx)$. If $y\in N$  and  for $\forall n\in \mathbb{Z}$, 
	\[d(f^nx,f^ny)\leq r_n(x)=Q(f^nx).\]
	then we have $f^ny\in \exp_{f^nx}(S_{f^nx})$. Therefore, $y$ shadows the pointwise pseudo orbit $\{f^nx\}_{n\in \mathbb{Z}}$.
	Due to the uniqueness of shadowing, we get $x = y$. Thus, $f$ is pointwise expansive on $N$.
\end{proof}

\vspace{1cm}
\subsection{The preservation of pointwise hyperbolicity under perturbations}
\begin{lem}\cite{SRB}\label{锥定理}
	Let $f:M\rightarrow M $ be a $C^1$ diffeomorphism and $S\subset M$ satisfy $f(S) = S$ and $O = M\setminus S$. If for every $x\in O$, there exist two continuous cone families $\{C_x^u\}$ and $\{C_x^s\}$ such that:\\
	(1)$\quad D_xf(C_x^u)\subseteq C_{fx}^u,\quad D_xf(C_x^s)\supseteq C_{fx}^s$\\
	(2)$\quad\|D_xfv\|>\|v\| ~~\forall v\in C_x^u;~~\|D_xfv\|<\|v\| ~~\forall v\in C_x^s.$\\
	(3)$\quad\lim_{n\to \infty}\|D_xf^nv\|=\infty ~~\forall v\in C_x^u;
	\quad\lim_{n\to \infty}\|D_xf^{-n}v\|=\infty \quad\forall v \in C_x^s,$\\
	then $f$ is pointwise hyperbolic.
\end{lem}
\noindent\textbf{Assumption K:} There exists $\varepsilon_0 > 0$ such that for all $\varepsilon\leq\varepsilon_0$, we have 
$$\frac{\|Df|_{E^s(x)}\|}
{m(Df|_{E^u(x)})}
<\frac{\|Df|_{E^s(x)}\|+\varepsilon\varepsilon(x)+1}
{\|Df|_{E^s(f(x))}\|+\varepsilon\varepsilon(f(x))+1}
\cdot \frac{m(Df|_{E^u(f(x))})-\varepsilon\varepsilon(f(x))-1}{m(Df|_{E^u(x)})-\varepsilon\varepsilon(x)-1}.$$

Since as $x\to \partial N$, $\|Df|_{E^s(x)}\|+\varepsilon\varepsilon(x)\to 1$ and $m(Df|_{E^u(x)})-\varepsilon\varepsilon(x)\to 1$. In fact, to some extent, Assumption K describes  the relationship between $\frac{\|Df|_{E^s(x)}\|}{m(Df|_{E^u(x)})}$ and the rate at which the point approaches the boundary of $N$. 
Next, we consider a slight perturbation $g$ of $f$ on $N$ under the assumption K.
\vspace{0.25cm}
\begin{them}\label{保持逐点双曲}
	If $f$ is pointwise hyperbolic on the open set $N$ and satisfies assumptions U, S, R and K, then there exists a continuous function $\xi:N\to\mathbb{R}_{+}$ such that when $\varepsilon > 0$ is sufficiently small, the perturbed diffeomorphism $g$ satisfying the following conditions is also pointwise hyperbolic on $N$:
	for $\forall x\in N$,\\
	(1)\quad$\prod_{i=0}^\infty m(Df|_{E^u(g^ix)})=\infty
	, \prod_{i=0}^\infty m(Df^{-1}|_{E^s(g^{-i}x)})=\infty$\\
	(2)\quad$\max\{d(f(x),g(x)),\| D_xf-D_xg\|\}<\xi(x)\delta^u(x)Q(x)^2Q(f(x))$
\end{them}
\begin{proof}
	Recall that $\delta^u(x)<\min\{\varepsilon\varepsilon(x),\rho - \sqrt{2}C^fQ(x)\}$. When $$d(f(x),y)<\delta^u(x)Q(x)^2Q(f(x)),  d(f^{-1}(y),x)<\delta^s(y)Q(y)^2Q(f^{-1}(y)),$$ there are corresponding graph transform and pointwise pseudo orbits are defined.
	Let $\|\cdot\|$ be the norm induced by the Riemannian metric, and $\|\cdot\|_{\triangle}$ be its box norm under $E^u\oplus E^s$. Choose two continuous cone families of $f$:
	$$C_x^u=\{v\in T_xM:\|v_s\|<\kappa_x^u\|v_u\|\}$$
	$$C_x^s=\{v\in T_xM:\|v_u\|<\kappa_x^u\|v_s\|\}$$
	where $$\kappa_x^u=\min\{1,\frac{m(Df|_{E^u(x)})-\varepsilon\varepsilon(x)-1}{\|Df|_{E^s(x)}\|+\varepsilon\varepsilon(x)+1}\},$$
	$$\kappa_x^s=\min\{1,\frac{m(Df^{-1}|_{E^s(x)})-\varepsilon\varepsilon(x)-1}{\|Df^{-1}|_{E^u(x)}\|+\varepsilon\varepsilon(x)+1}\}.$$ Next, we will prove that $g$ is pointwise hyperbolic by following the method of Lemma \ref{锥定理}.  In fact, we only discuss the unstable cone $C_x^u$, and the case of the stable cone $C_x^s$ is similar.
	\begin{enumerate}
		\item[(i)]
		$ D_xg(C_x^u)\subseteq C_{g(x)}^u,D_xg(C_x^s)\supseteq C_{g(x)}^s$
		\item[(ii)]
		$\|D_xgv\|_\triangle>\|v\|_\triangle \quad\forall v\in C_x^u; \quad\|D_xgv\|_\triangle<\|v\|_\triangle \quad\forall v\in C_x^s.$
		\item[(iii)]
		$\lim_{n\to \infty}\|D_xg^nv\|_\triangle=\infty \quad\forall v\in C_x^u;
		\quad\lim_{n\to  \infty} \|D_xg^{-n}v\| _\triangle=\infty \quad\forall v \in C_x^s.$
	\end{enumerate}
	\textbf{Step 1:} First, it can be proven that $D_xg(C_x^u)\subseteq C_{gx}^u$.
	Notice that with respect to \(E^u\oplus E^s\), for \(\forall v=(v_u,v_s)\in C_x^u\),
	\begin{equation}\label{可扩1}
		\begin{aligned}
			\frac{\|(D_xgv)_s\|}{\|(D_xgv)_u\|}&=\frac{\|(D_xgv_s)_s+(D_xgv_u)_s\|}{\|(D_xgv_u)_u+(D_xgv_s)_u\|}\\
			&\leq\frac{\|(D_xg)_{ss}v_s\|+\|(D_xg)_{su}v_u\|}{\|(D_xg)_{uu}v_u\|-\|(D_xg)_{us}v_s\|}\\
			&\leq\frac{\|(D_xg)_{ss}\|\|v_s\|+\|(D_xg)_{su}\|\|v_u\|}{m((D_xg)_{uu})\|v_u\|-\|(D_xg)_{us}\|\|v_s\|}\\
			&\leq\frac{\|(Dg)_{ss}\|\kappa_x^u\|v_u\|+\|(D_xg)_{su}\|\|v_u\|}{m((D_xg)_{uu})\|v_u\|-\|(D_xg)_{us}\|\kappa_x^u\|v_u\|}\\
			&=\frac{\|(D_xg)_{ss}\|\kappa_x^u+\|(D_xg)_{su}\|}{m((D_xg)_{uu})-\|(D_xg)_{us}\|\kappa_x^u}
		\end{aligned}
	\end{equation}
	Since $\frac{\|Df|_{E^s(x)}\|}
	{m(Df|_{E^u(x)})}=\frac{\|(D_xf)_{ss}\|}{m((D_xf)_{uu})},$	by Assumption K,
	\begin{equation}\label{可扩2}
		\begin{aligned}
			\frac{\|(D_xf)_{ss}\|}{m((D_xf)_{uu})}\kappa_x^u
			&\leq \frac{\|Df|_{E^s(x)}\|}
			{m(Df|_{E^u(x)})}\cdot\frac{m(Df|_{E^u(x)})-\varepsilon\varepsilon(x)-1}{\|Df|_{E^s(x)}\|+\varepsilon\varepsilon(x)+1}\\
			&<\min\{1,\frac{m(Df|_{E^u(f(x))})-\varepsilon\varepsilon(f(x))-1}{\|Df|_{E^s(f(x))}\|+\varepsilon\varepsilon(f(x))+1}\}\\
			&=\kappa_{f(x)}^u 
		\end{aligned}
	\end{equation}
	Denote $\rho_x=\kappa_{f(x)}^u-\frac{\|(D_xf)_{ss}\|}{m((D_xf)_{uu})}\kappa_x^u$. Then there exists a function $\xi(x)>0$ which is small enough such that when
	$$\max\{d(f(x),g(x)),\|D_xf - D_xg\|\}<\xi(x)\delta^u(x)Q(x)^2Q(f(x)),$$
	then $\|(Dg)_{su}\|$ and $\|(D_xg)_{us}\|$ are arbitrarily small, and $\|(D_xg)_{ss}\|$, $m((D_xg)_{uu})$ are arbitrarily close to $\|(D_xf)_{ss}\|$, $m((D_xf)_{uu})$ respectively, so that 
	$$\frac{\|(D_xg)_{ss}\|\kappa_x^u+\|(D_xg)_{su}\|}{\|m((D_xg)_{uu})-\|(D_xg)_{us}\|\kappa_x^u}\leq	\frac{\|(D_xf)_{ss}\|}{m((D_xf)_{uu})}\kappa_x^u+\frac{1}{3}\rho_x,$$
	$$\frac{m(Df|_{E^u(g(x))})-\varepsilon\varepsilon(g(x))-1}{\|Df|_{E^s(g(x))}\|+\varepsilon\varepsilon(g(x))+1}\geq\frac{m(Df|_{E^u(f(x))})-\varepsilon\varepsilon(f(x))-1}{\|Df|_{E^s(f(x))}\|+\varepsilon\varepsilon(f(x))+1}-\frac{1}{3}\rho_x.$$
	Therefore, according to (\ref{可扩1}),(\ref{可扩2}), we have
	
	\begin{equation*}
		\begin{aligned}
			\frac{\|(D_xgv)_s\|}{\|(D_xgv)_u\|}
			&\leq\frac{\|(D_xg)_{ss}\|\kappa_x^u+\|(D_xg)_{su}\|}{\|m((D_xg)_{uu})-\|(D_xg)_{us}\|\kappa_x^u}\\
			&\leq\frac{\|(D_xf)_{ss}\|}{m((D_xf)_{uu})}\kappa_x^u+\frac{1}{3}\rho_x\\
			&<\kappa_{f(x)}^u-\frac{1}{3}\rho_x\\
			&=\min\{1,\frac{m(Df|_{E^u(f(x))})-\varepsilon\varepsilon(f(x))-1}{\|Df|_{E^s(f(x))}\|+\varepsilon\varepsilon(f(x))+1}\}-\frac{1}{3}\rho_x\\
			&\leq\min\{1,\frac{m(Df|_{E^u(g(x))})-\varepsilon\varepsilon(g(x))-1}{\|Df|_{E^s(g(x))}\|+\varepsilon\varepsilon(g(x))+1}\}\\
			&=\kappa_{g(x)}^u.
		\end{aligned}
	\end{equation*}
	That is, $D_xg(v)\in C_{gx}^u$. Due to the arbitrariness of $v$, we can conclude that $D_xg(C_x^u)\subseteq C_{gx}^u$.\\
	\textbf{Step 2:}
	It can be shown that for all $x\in N$ and all $v\in C_x^u$, $\|D_xgv\|_{\triangle}>\|v\|_{\triangle}$. Furthermore, \(\lim_{n\rightarrow\infty}\|D_xg^nv\|_{\triangle}=\infty\).
	
	For all $v\in C_x^u$, we have $\|D_xfv\|_{\triangle}\geq m(Df|_{E^u(x)})\|v\|_{\triangle}$. Moreover,
	let $$a_2 = \inf_{x\in N}\frac{m(Df|_{E^u(x)})-\sqrt{m(Df|_{E^u(x)})}}{m(Df|_{E^u(x)}) - 1}.$$
	Take $\xi(x)<1$ and $\varepsilon$ sufficiently small such that
$$\max\{d(f(x),g(x)),\|D_xf - D_xg\|\}<\xi(x)\delta^u(x)Q(x)^2Q(f(x))\leq\varepsilon\varepsilon(x).$$
	Moreover, if $\varepsilon < a_2C_0$, then
	\begin{equation*}
		\begin{aligned}
			\|D_xgv\|_\triangle
			&\geq(m(Df|_{E^u(x)})-\varepsilon\varepsilon(x))\|v\|_\triangle\\
			&\geq(m(Df|_{E^u(x)})-a_2C_0\varepsilon(x))\|v\|_\triangle \\
			&\geq\sqrt{m(Df|_{E^u(x)})}\|v\|_\triangle>\|v\|_\triangle,
		\end{aligned}	
	\end{equation*}
	then
	$$\lim_{n\to \infty}\|D_xg^nv\|_\triangle\geq \prod_{i=0}^\infty\sqrt {m(Df|_{E^u(g^ix)})}\|v\|_\triangle.$$
	Since $g$ satisfies $\prod_{i = 0}^{\infty}m(Df|_{E^u(g^ix)})=\infty$, we have $\lim_{n\to \infty}\|D_xg^nv\|_\triangle=\infty.$\\
	\textbf{Step 3:}
	According to (i)-(iii) above, $G_x^u=\bigcap_{n = 0}^{\infty}D_{g^{-n}x}g^n(C_{g^{-n}x}^u)$ is the expanding subspace of $g$.
	In fact, $G_x^u=\bigcap_{n = 0}^{\infty}D_{g^{-n}x}g^n(C_{g^{-n}x}^u)$ is an invariant subspace of $T_xM$.
	By contradiction, if it is not, then there exist $v_1,v_2\in G_x^u$, while $v_1 - v_2\in C_x^s$. So we have 
	\begin{equation}\label{可扩3}
		\lim_{n\to  \infty} \|D_xg^{-n}(v_1-v_2)\| _\triangle=\infty.
	\end{equation}
	On the other hand, since $G_x^u\subseteq D_{g^{-n}x}g^n(C_{g^{-n}x}^u)$, for all $n \in\mathbb{Z}$, we have $v_1,v_2\in D_{g^{-n}x}g^n(C_{g^{-n}x}^u)$.
	That is, $D_xg^{-n}(v_1),D_xg^{-n}(v_2)\in C_{g^{-n}x}^u$.
	Then $$ \|D_xg^{-n}(v_1-v_2)\|_\triangle\leq\|D_xg^{-n}(v_1)\|_\triangle+\|D_xg^{-n}(v_2)\|_\triangle	<\|v_1\|_\triangle+\|v_2\|_\triangle
	<\infty
	$$
	This is in contradiction with (\ref{可扩3}). Moreover, by $D_xg(G_x^u)=\bigcap_{n = 0}^{\infty}D_{g^{-n}x}g^{n + 1}(C_{g^{-n}x}^u)=G_{g(x)}^u$, we can see that $G_x^u$ is an invariant subspace of $Dg$.
	Next, we will show that for all $v\in C_x^u$, $\|D_xgv\|>\|v\|$. This is because 
	\begin{equation}\label{可扩4}
		\begin{aligned}
			\|D_xgv\|&=\|D_xg(v_s+v_u)\|\\
			&\geq\|D_xgv_u\|-\|D_xgv_s\|\\
			&\geq m(Dg|_{E^u(x)})\|v_u\|-\|Dg|_{E^s(x)}\|\|v_s\|\\
			&\geq( m(Df|_{E^u(x)})-\varepsilon\varepsilon(x))\|v_u\|-(\|Df|_{E^s(x)}\|+\varepsilon\varepsilon(x))\|v_s\|.
		\end{aligned}
	\end{equation}
	Since $v\in C_x^u$, then we have
	\begin{equation}\label{可扩5}
		\|v_s\|<\kappa_x^u\|v_u\|
		\leq \frac{m(Df|_{E^u(x)})-\varepsilon\varepsilon(x)-1}{\|Df|_{E^s(x)}\|+\varepsilon\varepsilon(x)+1}\|v_u\|.
	\end{equation}
	By (\ref{可扩4}),(\ref{可扩5}),
	$$\|D_xgv\|\geq( m(Df|_{E^u(x)}) -\varepsilon\varepsilon(x))\|v_u\|-(\|Df|_{E^s(x)}\|+\varepsilon\varepsilon(x))\|v_s\|	>\|v_u\|+\|v_s\|\geq\|v\|$$
	Since $G_x^u\subseteq C_x^u$, $Dg$ is expansive on $G_x^u$. So far, we have proved that $G_x^u$ is an expansive invariant subspace of $T_xM$.
	Similarly, we can obtain that $G_x^s$ is a contractive invariant subspace of $T_xM$. Furthermore, because $G_x^u\cap G_x^s = \{0\}$ and their dimensions are complementary, we have $T_xM = G_x^u\oplus G_x^s$, which is the hyperbolic decomposition of $g$. That is, $g$ is pointwise hyperbolic on $N$.
\end{proof}
\subsection{ The stability of pointwise hyperbolic systems}
 \begin{lem}\label{Q(x)接近}
	For any sufficiently small $\varepsilon>0$, when $d(x,y)\leq Q(x)$, we have 
	$$\frac{1}{2}Q(x)\leq Q(y)\leq2Q(x).$$
\end{lem}
In fact, when $d(x,y)\leq Q(x)=\varepsilon^{\frac{2}{\alpha - \delta}}\min\{r_0^{\gamma},d(x,\partial N)^{\gamma}\}$, for sufficiently small $\varepsilon>0$, $d(x,\partial N)$ and $d(y,\partial N)$ are approximately equal, which means that $Q(x)$ and $Q(y)$ are approximately equal.
\begin{proof}Specifically, it can be divided into the following cases.\\
	\textbf{Case 1:} When $d(x,\partial N)\geq r_0$ and $d(y,\partial N)\geq r_0$, we have $Q(x) = Q(y)=\varepsilon^{\frac{2}{\alpha - \delta}}r_0^{\gamma}$, so the inequality $\frac{1}{2}Q(x)\leq Q(y)\leq2Q(x)$ obviously holds.
	\\\textbf{Case 2:} When $d(y,\partial N)\leq r_0\leq d(x,\partial N)$, we have $Q(y)=\varepsilon^{\frac{2}{\alpha - \delta}}d(y,\partial N)^{\gamma}$ and $Q(x)=\varepsilon^{\frac{2}{\alpha - \delta}}r_0^{\gamma}$.
	Notice that $Q(y)=\varepsilon^{\frac{2}{\alpha - \delta}}d(y,\partial N)^{\gamma}\leq\varepsilon^{\frac{2}{\alpha - \delta}}r_0^{\gamma} = Q(x)$, so $Q(y)\leq 2Q(x)$ clearly holds.
	On the other hand, to prove $Q(x)\leq 2Q(y)$, we only need to prove $r_0\leq 2^{\frac{1}{\gamma}}d(y,\partial N)$.
	Since
	\begin{equation*}
		\begin{aligned}
			d(y,\partial N)&\geq d(x,\partial N)-d(x,y)\\
			&\geq d(x,\partial N)-Q(x) \\
			&=d(x,\partial N)-\varepsilon^\frac{2}{\alpha-\delta}r_0^\gamma\\
			&\geq(1-\varepsilon^\frac{2}{\alpha-\delta}) r_0,  
		\end{aligned}
	\end{equation*}
	when $\varepsilon<(1-2^{-{\frac{1}{\gamma}}})^{\frac{\alpha-\delta}{2}}$, we have $r_0\leq 2^\frac{1}{\gamma} d(y,\partial N) $.\\
	\textbf{Case 3:} When $d(x,\partial N)\leq r_0\leq d(y,\partial N)$, we have $Q(x)=\varepsilon^{\frac{2}{\alpha - \delta}}d(x,\partial N)^{\gamma}$ and $Q(y)=\varepsilon^{\frac{2}{\alpha - \delta}}r_0^{\gamma}$.
	Notice that $Q(x)=\varepsilon^{\frac{2}{\alpha - \delta}}d(x,\partial N)^{\gamma}\leq\varepsilon^{\frac{2}{\alpha - \delta}}r_0^{\gamma}=Q(y)$, so $Q(x)\leq 2Q(y)$ clearly holds.
	Similarly, to prove $Q(y)\leq 2Q(x)$, we only need to prove $r_0\leq 2^{\frac{1}{\gamma}}d(x,\partial N)$.
	Since $d(x,\partial N)< 1$ and $\gamma>1$, we know that $d(x,\partial N)^{\gamma}<d(x,\partial N)$, so 
	\begin{equation*}
		\begin{aligned}
			r_0&\leq d(y,\partial N)\leq d(x,\partial N)+d(x,y)\\
			&\leq d(x,\partial N)+Q(x)\\&\leq(1+\varepsilon^\frac{2}{\alpha-\delta})d(x,\partial N).
		\end{aligned}
	\end{equation*}
	Take $\varepsilon<(2^{\frac{1}{\gamma}} - 1)^{\frac{\alpha-\delta}{2}}$, then $r_0\leq 2^{\frac{1}{\gamma}}d(x,\partial N)$.\\
	\textbf{Case 4:} When $d(x,\partial N)\leq r_0$ and $d(y,\partial N)\leq r_0$, we have $Q(x)=\varepsilon^{\frac{2}{\alpha - \delta}}d(x,\partial N)^{\gamma}$ and $Q(y)=\varepsilon^{\frac{2}{\alpha - \delta}}d(y,\partial N)^{\gamma}$. To prove $Q(y)\leq 2Q(x)$, we only need to prove $d(y,\partial N)\leq 2^{\frac{1}{\gamma}}d(x,\partial N)$. Since
	\[d(y,\partial N)\leq d(x,\partial N)+d(x,y)\leq d(x,\partial N)+Q(x)\leq(1 + \varepsilon^{\frac{2}{\alpha - \delta}})d(x,\partial N),\]
	Similar to Case 3, by taking $\varepsilon<(2^{\frac{1}{\gamma}} - 1)^{\frac{\alpha - \delta}{2}}$, we can get $Q(y)\leq 2Q(x)$. To prove $Q(x)\leq 2Q(y)$, we only need to prove $d(x,\partial N)\leq 2^{\frac{1}{\gamma}}d(y,\partial N)$.Since
	\begin{equation*}
		\begin{aligned}
			d(y,\partial N)&\geq d(x,\partial N)-d(x,y)\\
			&\geq d(x,\partial N)-Q(x) \\
			&=d(x,\partial N)-\varepsilon^\frac{2}{\alpha-\delta}d(x,\partial N)^\gamma\\
			&\geq(1-\varepsilon^\frac{2}{\alpha-\delta}) d(x,\partial N),  
		\end{aligned}
	\end{equation*}
	similar to Case 2, if we take $\varepsilon<(1 - 2^{-\frac{1}{\gamma}})^{\frac{\alpha-\delta}{2}}$, then $Q(x)\leq 2Q(y)$.
\end{proof}
\begin{lem}\label{证h连续}
	For every $x\in N$ and every $\lambda>0$, there exists $K(x,\lambda)\in\mathbb{N}$ such that if $y\in N$ satisfies
	$$d(f^nx,f^ny)\leq Q(f^nx)\quad\text{for all}\quad |n|\leq K(x,\lambda),$$
	then $d(x,y)<\lambda$.
\end{lem}
\begin{proof}
	Proof by contradiction: Assume that the statement does not hold, then for some $x\in N$ and $\lambda > 0$, for all $K\geq1$, there exists $y_K\in N$ such that when $d(f^nx,f^ny_K)\leq Q(f^nx)$ for all $|n|\leq K$, we have $d(x,y_K)\geq\lambda$.
	Since $M$ is compact, we can take a sequence $\{K_i\}_{i\rightarrow\infty}$ such that $y_{K_i}\to y\in\overline{N}$ , then  $$d(f^nx,f^ny_{K_i})\leq Q(f^nx)~~\forall |n|\leq K_i\Rightarrow d(x,y_{K_i})\geq\lambda .$$
	Therefore, if for all $n \in\mathbb{Z}$, $d(f^nx,f^ny)\leq Q(f^nx)$, then $d(x,y)\geq\lambda$. Notice that $d(x,y)\leq Q(x)<d(x,\partial N)$, so $y\in N$. According to the pointwise expansivity of $f$, for any $x,y\in N$, when for all $n$, $d(f^nx,f^ny)\leq Q(f^nx)$, then $x = y$. This leads to a contradiction.
\end{proof}

\begin{lem}\cite{Mun}\label{满}
	Let $M$ and $N$ be topological manifolds, and $ f: M\rightarrow N$ be a homeomorphism between the two topological manifolds. Then there exists a positive continuous function $\psi(x)$ such that if $d(f(x), g(x))<\psi(x)$ and $g(\text{Bd}(M))\subset\text{Bd}(N)$, then $g$ is surjective.
\end{lem}

\begin{them}\label{稳定性}
	If $f:M\to M$ is pointwise hyperbolic on the invariant set $N$ and satisfies the assumptions U, S, R, then for any sufficiently small $\varepsilon>0$, $f$ is "semi - pointwise quasi - stable", that is, if the diffeomorphism $g:M\to M$ satisfies: \\
	(1)\quad$f|_{M\backslash N}=g|_{M\backslash N}$\\
	(2)\quad$\prod_{i=0}^\infty m(Df|_{E^u(g^ix)})=\infty
	,  \prod_{i=0}^\infty m(Df^{-1}|_{E^s(g^{-i}x)})=\infty  \quad \forall x\in N$\\
	(3)\quad$\max\{d(f(x),g(x)),\| D_xf-D_xg\|\}<\xi(x)\delta^u(x)Q(x)^2Q(f(x)) \quad \forall x\in N$
	\\ then there exists a continuous surjective map $h:M\to M$ such that $h\circ g = f\circ h$. In particular, $h|_{M\setminus N}=\text{id}$ and $\forall x\in N, d(h(x),x)\leq Q(x)$.
	
	Furthermore, when $f$ also satisfies assumption K, if there exists a sequence of functions $\{r_g^n(x)\}_{n\in\mathbb{Z}}$ such that "for all $n$, when $d(g^nx, g^ny)\leq r_g^n(x)$, then $x = y$," then $h$ is a homeomorphism. In this case, $f$ is said to be "pointwise quasi - stable".
\end{them}
\begin{proof}
	
	\textbf{Step 1:} First, we prove that there exists a map $h$ on $N$ such that $h\circ g = f\circ h$ and $d(h(x),x)\leq Q(x)$.
	
	According to the definition of pointwise pseudo orbit, for $\forall x\in N$, $\{g^n(x)\}_{n\in\mathbb{Z}}$ is a pointwise pseudo orbit of $f$. In fact, from (3), for sufficiently small  $\xi(x)$, when 
	$$\max\{d(f(x),g(x)),\| D_xf - D_xg\|\}<\xi(x)\delta^u(x)Q(x)^2Q(f(x))$$
	we have
	$$\max\{d(f^{-1}(x),g^{-1}(x)),\| D_xf^{-1}-D_xg^{-1}\|\}<\delta^s(x)Q(x)^2Q(f^{-1}(x)),$$
	then
	$$d(f(g^nx),g^{n + 1}x)<\delta^u(g^nx)Q(g^nx)^2Q(fg^nx)$$
	$$d(f^{-1}(g^{n+1}x),g^{n}x)<\delta^s(g^{n+1}x)Q(g^{n+1}x)^2Q(f^{-1}g^{n+1}x).$$
	From (2), we know that 
	$$\prod_{i=0}^\infty\|Df^{-1}|_{E^u(fg^i(x))}\|=0,\quad \prod_{i=-\infty}^0\|Df|_{E^s(f^{-1}g^{i}(x))}\|=0\quad\forall x\in N.$$
	Since $E^s$ and $E^u$ are continuous, if $\xi$ is sufficiently small, then we have
	$$\left\|Df^{-1}\big|_{E^u(g^{i + 1}(x))}\right\|\leq\sqrt{\left\|Df^{-1}\big|_{E^u(fg^i(x))}\right\|}, \quad\left\|Df\big|_{E^s(g^{i - 1}(x))}\right\|\leq\sqrt{\left\|Df\big|_{E^s(f^{-1}g^{i}(x))}\right\|}$$
	Furthermore, we have:	$$\prod_{i=0}^\infty\|Df^{-1}|_{E^u(g^i(x))}\|=0,\quad \prod_{i=-\infty}^0\|Df|_{E^s(g^{i}(x))}\|=0\quad\forall x\in N.$$
	According to the definition of the pointwise pseudo orbit, $\{g^n(x)\}_{n\in\mathbb{Z}}$ is a pointwise pseudo orbit of $f$.

	According to the shadowing Lemma \ref{跟踪引理}, for $\forall x\in N$, there exists a unique point $h(x)\in N$ such that the sequence $\{g^n(x)\}_{n\in\mathbb{Z}}$ is shadowed by $h(x)$, that is, for $\forall n \in\mathbb{Z}$,
	\[f^nh(x)\in \exp_{g^nx}(S_{g^nx}).\]
	This determines a map $h:N\to N$. Since for all $n \in\mathbb{Z}$,
	\[f^{n + 1}h(x)\in \exp_{g^{n+1}x}(S_{g^{n+1}x})\quad\text{and}\quad f^nh(g(x))\in \exp_{g^{n+1}x}(S_{g^{n+1}x}),\]
	 then both $fh(x)$ and $hg(x)$ shadow the pointwise pseudo orbit $\{g^{n+1}(x)\}_{n\in\mathbb{Z}}$.
	By the uniqueness of shadowing, we have $h\circ g = f\circ h$. According to Lemma \ref{跟踪引理}, since $h(x)$ shadows $\{g^n(x)\}_{n\in\mathbb{Z}}$, we have
	\[d(h(x),x)\leq\frac{1}{50}Q(x)<Q(x).\]\\
	\textbf{Step 2:} We will prove that $h:N\to N$ is continuous.
	For all $x\in N$ and all $\lambda > 0$, by Lemma \ref{证h连续}, $\exists K:= K(h(x),\lambda)$ such that for all $z\in N$, then
	\begin{equation}\label{稳定性1}
		d(f^nh(x),f^nz)\leq Q(f^nh(x))~~\forall |n|\leq K\Rightarrow d(h(x),z)<\lambda.
	\end{equation}
	Since $g$ is continuous, there exists a $\delta$ such that when $y$ satisfies $d(x,y)<\delta$, 
	\begin{equation}\label{稳定性2}
		d(g^nx,g^ny)\leq \frac{1}{25}Q(g^nx)~~\forall |n|\leq K
	\end{equation}
	by Lemma\Ref{Q(x)接近}, we have
	\begin{equation}\label{稳定性3}
		Q(g^ny)\leq2Q(g^nx).
	\end{equation}
	Moreover, since $h(x)$ shadows $\{g^n x\}_{n\in\mathbb{Z}}$, we have $d(hg^n x, g^n x)\leq Q(g^n x)$. According to Lemma \Ref{Q(x)接近}, we also have
	\begin{equation}\label{稳定性4}
		Q(hg^nx)\leq2Q(g^nx).
	\end{equation}
	Then for all $|n|\leq K ,$ by (\ref{稳定性2}),(\ref{稳定性3}),(\ref{稳定性4}),We can obtain that
	\begin{equation*}
		\begin{aligned}
			d(f^nhx,f^nhy)&=d(hg^nx,hg^ny)\\
			&\leq{}d(hg^nx,g^nx)+d(hg^ny,g^ny)+d(g^nx,g^ny) \\
			&\leq\frac{1}{50}(Q(g^nx)+Q(g^ny))+\frac{1}{25}Q(g^nx)\\
			&\leq\frac{3}{50}Q(g^nx)+\frac{1}{25}Q(g^nx)\\
			&\leq\frac{1}{10}Q(g^nx)\\
			&\leq Q(hg^nx)=Q(f^nhx).   
		\end{aligned}
	\end{equation*}
	Therefore, according to (\ref{稳定性1}), $d(h(x), h(y)) < \lambda$. Thus, $h$ is continuous at $x$. By the arbitrariness of $x$, we can conclude that $h:N\to N$ is continuous.\\
	\textbf{Step 3:} We will prove that $h:N\to N$ is surjective.
	In the previous discussion, we have shown that $h:N\to N$ is a continuous map, and for all $x\in N$, \[d(h(x),x)\leq Q(x)=\varepsilon^\frac{2}{\alpha-\delta}\min\{r_0^\gamma,d(x,\partial N)^\gamma\}\]
	According to the definition of $Q(x)$, we naturally extend $Q(x)$ to $\partial N$. That is, for all $x\in\partial N$, $Q(x) = 0$.
	Now, we continuously extend $h(x)$ to $\partial N$. For all $x\in\partial N$, we take $\{x_n\}\subset N$ such that $x_n\to x$, and define $h(x):=\lim_{n\to\infty}h(x_n)$.
	Notice that as $x_n\to x\in\partial N$, from the inequality $d(h(x_n),x_n)\leq Q(x_n)$, taking the limit on both sides with respect to $n$, we get $d(h(x),x) = 0$. That is, $h(x)=x$ for all $x\in\partial N$.
	Therefore, $h(\partial N)=\partial N$.

	According to Lemma \ref{满}, the identity map $\text{id}:\overline{N}\to\overline{N}$ is a homeomorphism. Thus, there exists a positive continuous function $\psi(x)$ on $\overline{N}$ such that if $h$ satisfies $d(h(x),\text{id}(x))<\psi(x)$ and $h(\text{Bd}(\overline{N}))\subset\text{Bd}(\overline{N})$, then $h:\overline{N}\to\overline{N}$ is surjective.
	Since $h(\partial N)=\partial N$, it is obvious that $h(\text{Bd}(\overline{N}))\subset\text{Bd}(\overline{N})$. Because $\overline{N}$ is compact and $\psi(x)$ is a positive continuous function on $\overline{N}$, there exists $\delta_0 > 0$ such that for all $x\in\overline{N}$, $\delta_0\leq\psi(x)$.
	Notice that$$d(h(x),x)\leq Q(x)\leq\varepsilon^\frac{2}{\alpha-\delta}r_0^\gamma,$$ so when $\varepsilon$ is small enough, we have $$d(h(x),x)\leq Q(x)\leq\varepsilon^\frac{2}{\alpha-\delta}r_0^\gamma<\delta_0\leq \psi(x).$$
	Therefore, $h:\overline{N}\to\overline{N}$ is surjective. Since $h(\partial N)=\partial N$, then $h:N\to N$ is surjective.
	
	Note that we can further continuously extend $h$ to the outside of $N$. Let $h|_{M\setminus N}=\text{id}$ and $Q|_{M\setminus N}=0$, we still have $h\circ g = f\circ h$ and $d(h(x),x)\leq Q(x)$.
	So far, there exists a continuous surjective mapping $h:M\to M$ that satisfies the conditions. In this case, $f$ is "semi - pointwise quasi - stable".\\
	\textbf{Step 4:} We will prove that when $f$ satisfies Assumption K, if there exists a sequence of functions $\{r_g^n(x)\}_{n\in\mathbb{Z}}$ such that "for all $n$, when $d(g^nx,g^ny)\leq r_g^n(x)$, then $x = y$", then $h$ is a homeomorphism.
	
	According to Theorem \ref{逐点可扩}, we know that $f$ is pointwise expansive. That is, for all $x\in N $, if $y\in N$ satisfies $d(f^nx,f^ny)\leq Q(f^nx)$ for all $n \in\mathbb{Z}$, then $x = y$.
	According to Lemma \ref{保持逐点双曲}, if $f$ satisfies assumption K, then the perturbation $g$ is also pointwise hyperbolic. Due to the preservation of pointwise hyperbolicity under perturbation, $g$ also implies a certain expansivity property. At this time, if there exists a sequence of functions $\{r_g^n(x)\}$ such that
	\begin{equation}\label{单射1}
		d(g^nx,g^ny)\leq r_g^n(x)\quad \forall n \Rightarrow x=y.
	\end{equation}
then we can prove that $h:N\to N$ is injective.
	
	In fact, we prove it by contradiction. Suppose that $h:N\to N$ is not injective. That is, there exist $x_0,y_0\in N$ with $x_0\neq y_0$, but $h(x_0) = h(y_0)$.
	According to (\ref{单射1}), for $x_0\neq y_0$, there exists an $n_0$ such that
	\begin{equation}\label{单射2}
		d(g^{n_0}x_0,g^{n_0}y_0)> r_g^{n_0}(x_0).
	\end{equation}
	On the other hand, since $h(x_0)=h(y_0) $, then
	\begin{equation*}
		\begin{aligned}
			d(g^{n_0}x_0,g^{n_0}y_0)&\leq d(hg^{n_0}x_0,g^{n_0}x_0)+d(hg^{n_0}y_0,g^{n_0}y_0)+d(hg^{n_0}x,hg^{n_0}y_0)\\
			&\leq d(hg^{n_0}x_0,g^{n_0}x_0)+d(hg^{n_0}y,g^{n_0}y_0)+d(f^{n_0}hx_0,f^{n_0}hy_0)\\
			&=d(hg^{n_0}x_0,g^{n_0}x_0)+d(hg^{n_0}y_0,g^{n_0}y_0)\\
			&\leq\frac{1}{50}(Q(g^{n_0}x_0)+Q(g^{n_0}y_0))\\
			&\leq\frac{1}{25}\varepsilon^\frac{2}{\alpha-\delta}r_0^\gamma
		\end{aligned}
	\end{equation*}
	When $\varepsilon > 0$ is small enough, we have
	$$\frac{1}{25}\varepsilon^{\frac{2}{\alpha - \delta}}r_0^{\gamma}< r_g^{n_0}(x_0),$$
	that is, $d(g^{n_0}x_0,g^{n_0}y_0)<r_g^{n_0}(x_0)$, which contradicts (\ref{单射2}).
	Therefore, $h:N\to N$ is injective.
	
	At this point, after extending the map $h$, $h: M \to M$ is a homeomorphism. In this case, we say that $f$ is "pointwise quasi - stable". 
\end{proof}

\newpage
\addcontentsline{toc}{section}{References}

\end{document}